\documentclass[12pt]{amsart} 
\usepackage{amsfonts, amsmath,  amssymb, amsthm,  array, bm, multirow,
  pdfsync}

 \title[]{The topology of the set of nonsoliton Lie algebras in the 
moduli space of nilpotent Lie algebras}

\author{Tracy L. Payne}
\address{ Department of Mathematics, Idaho State University, 
 921 S. 8th Ave.,   Pocatello, ID 83209-8085}
\email{payntrac@isu.edu}


\newtheorem{thm}{Theorem}[section] 
\newtheorem{prop}[thm]{Proposition}
\newtheorem{lemma}[thm]{Lemma}  
\newtheorem{cor}[thm]{Corollary}

\theoremstyle{definition}  

\newtheorem{defn}[thm]{Definition}

\theoremstyle{remark}

\newcommand{\calB}{\ensuremath{\mathcal{B}} }

\newcommand{\calN}{\ensuremath{\mathcal{N}} }

\newcommand{\bfe}{\ensuremath{\bm e} }

\newcommand{\bfv}{\ensuremath{\bm v} }
\newcommand{\bfw}{\ensuremath{\bm w} }
\newcommand{\bfx}{\ensuremath{\bm x} }
\newcommand{\bfy}{\ensuremath{\bm y} }

\newcommand{\ric}{\mathsf{ric}}

\newcommand{\frakg}{\ensuremath{\mathfrak{g}} }
\newcommand{\frakh}{\ensuremath{\mathfrak{h}} }
\newcommand{\fraki}{\ensuremath{\mathfrak{i}} }
\newcommand{\frakj}{\ensuremath{\mathfrak{j}} }

\newcommand{\frakm}{\ensuremath{\mathfrak{m}} }
\newcommand{\frakn}{\ensuremath{\mathfrak{n}} }

\newcommand{\fraks}{\ensuremath{\mathfrak{s}} }
\newcommand{\frakt}{\ensuremath{\mathfrak{t}} }

\newcommand{\frakz}{\ensuremath{\mathfrak{z}} }

\newcommand{\boldN}{\ensuremath{\mathbb N}}

\newcommand{\boldR}{\ensuremath{\mathbb R}}

\newcommand{\ad}{\operatorname{ad}}

\newcommand{\Der}{\operatorname{Der}}

 \newcommand{\End}{\operatorname{End}}

 \newcommand{\Id}{\operatorname{Id}}

 \newcommand{\myspan}{\operatorname{span}}

 \newcommand{\rad}{\operatorname{rad}}

 \newcommand{\Ric}{\operatorname{Ric}}

 \newcommand{\trace}{\operatorname{trace}}

 \newcommand{\smallfrac}[2]{\textstyle{\frac{#1}{#2}}  }

\begin{document} 

  \begin{abstract} 
A Lie algebra is called {\em nonsoliton} if it does not admit a
soliton inner product.  
We demonstrate that the subset of nonsoliton Lie algebras in 
the moduli space of indecomposable $n$-dimensional
$\boldN$-graded nilpotent Lie algebras is discrete if and only if $n \le 7.$
 \end{abstract} 

  \maketitle


 \section{Introduction}\label{introduction}

A Lie algebra may be endowed with infinitely many different inner
products.
Among these, soliton inner products are considered preferred inner
products.  
An inner product $Q$ on a nilpotent Lie algebra $\frakg$ is called  {\em
soliton}
if the Ricci endomorphism $\Ric$ of $\frakg$ defined by $Q$
differs from a derivation of $\frakg$ by a scalar multiple of the
identity map on $\frakg.$
(See Section \ref{metric Lie algebras} for a precise definition of the
map $\Ric$).    
We will call a Lie algebra  {\em soliton} if it admits a soliton inner
product and we will call it  {\em nonsoliton} if it does not admit a
soliton inner product.  

Soliton inner products on nilpotent Lie algebras are called {\em
  nilsoliton}.  The study of nilsoliton inner products for nilpotent
Lie algebras originated in the analysis of Einstein solvmanifolds
(\cite{lauret01b}). Indeed, deep results of Heber and Lauret allow one
to reduce the study of Einstein inner products on solvable Lie
algebras to the study of soliton inner products on nilradicals
(\cite{heberinv}, \cite{lauret01b}, \cite{lauret-standard}).  Of
independent interest, a soliton inner product on a nilpotent Lie
algebra defines a metric on the corresponding simply connected
nilpotent Lie group that is soliton in the sense that the Ricci flow
moves the metric by diffeomorphisms and rescaling (\cite{lauret01b}).
And, outside of the category of homogeneous spaces, soliton inner
products are of use in a purely algebraic context in that they supply
extra structure for algebraic computations and may give canonical
presentations of Lie algebras (See Example 3.11 of \cite{payne-index}).

When they exist, nilsoliton inner products are unique up to scaling
(\cite{lauret01b}).  If a nilpotent Lie algebra admits a soliton inner
product, then it is $\boldN$-graded.  As not all nilpotent Lie
algebras are $\boldN$-graded, not all nilpotent Lie algebras admit
nilsoliton inner products.  One can find continuous families of
nonsoliton nilpotent Lie algebras by finding continuous families of
characteristically nilpotent Lie algebras.  (A Lie algebra
is {\em characteristically nilpotent} if its derivation algebra is nilpotent.) Such families exist in
dimensions seven and higher (see \cite{khakimdjanov02}).  As the direct
sum of  nilpotent Lie algebras is soliton if and only if each summand
is soliton (\cite{jablonski-09}, \cite{nikolayevsky-preEinstein}), we will restrict
our attention to indecomposable nilpotent Lie algebras.  We will study
the subset of nonsoliton Lie algebras in the moduli space of all
indecomposable $\boldN$-graded nilpotent Lie algebras of fixed
dimension.  In particular, we are interested in determining when the
set of nonsoliton Lie algebras is discrete in this moduli space.
  
In dimensions $6$ and lower, the situation is well-understood: the
moduli space of nilpotent Lie algebras is discrete
(\cite{degraaf-07}), and all nilpotent Lie algebras of dimension $6$
and lower admit soliton inner products (\cite{lauret02},
\cite{will03}).  In dimension $7,$ the moduli space of real nilpotent
Lie algebras consists of a finite number of discrete points and a
finite number of continuous families of nonisomorphic nilpotent Lie
algebras (\cite{seeley-93}, \cite{gong-98}).  Nikolayevsky proved that
if two real nilpotent Lie algebras have the same
complexification, then either both are soliton or both are nonsoliton
(\cite{nikolayevsky-preEinstein}).  This result was also proved
independently 
by M.\ Jablonski using different methods (\cite{jablonski-08b}).  Using this
result about complex forms, along with  Carles's classification of complex nilpotent Lie algebras
of dimension 7 (\cite{carles-96}, \cite{magnin-online}), Culma
determined precisely which $7$-dimensional complex nilpotent Lie
algebras have real forms that admit nilsoliton inner products
(\cite{culma1}, \cite{culma2}).  Culma found that among the continuous
families of nonsoliton nilpotent Lie algebras in dimension $7,$ none
of them are $\boldN$-graded.  It follows that the subset of nonsoliton
Lie algebras in the moduli space of indecomposable $7$-dimensional
$\boldN$-graded nilpotent Lie algebras is discrete.

Arroyo determined precisely which $\boldN$-graded filiform nilpotent
Lie algebras of dimension 8 admit soliton inner products
(\cite{arroyo-11}).  She found although there are continuous families
of solitons in that moduli space, there are precisely four isolated
nonsoliton Lie algebras.

Eberlein and Nikolayevsky showed independently that except for in two
cases, soliton Lie algebras are dense in the moduli space of two-step
nilpotent Lie algebras (\cite{eberlein-08},
\cite{nikolayevsky-preEinstein}).  (Note that all two-step nilpotent
Lie algebras are $\boldN$-graded by a derivation that equals the
identity on a complement to the center and that is twice the identity on
the center.)   Jablonski showed that the soliton Lie algebras are dense
in the remaining two cases-- when the nilpotent Lie algebra is of type 
$(2k+1,2)$ or $(2k+1,(\begin{smallmatrix}  2k+1 \\ 2  \end{smallmatrix})-2)$  (\cite{jablonski-thesis}).
\begin{thm}
[\cite{eberlein-08}, \cite{nikolayevsky-preEinstein}, \cite{jablonski-thesis}]
The set of soliton Lie algebras is dense in the 
moduli space of two-step
nilpotent Lie algebras.
\end{thm}
Given the classification results we have described, and the theorem
just stated, one might wonder if the set of nonsoliton Lie algebras is
always discrete in the moduli space of two-step $n$-dimensional
nonsoliton $\boldN$-graded nilpotent Lie algebras.  The answer is no.
The first continuous families of nonsoliton $\boldN$-graded nilpotent
Lie algebras were found by C.\ Will.
\begin{thm}[\cite{will-08}]
  The moduli space of indecomposable $9$-dimensional 
  two-step nilpotent Lie algebras contains two one-parameter families
  of nonsoliton nilpotent Lie algebras.
\end{thm}

Jablonski defined a general method for constructing families of
two-step nilpotent Lie algebras called concatenation.  He used the
concatenation construction to define continuous families of irreducible nonsoliton
two-step nilpotent Lie algebras in infinitely many dimensions.
\begin{thm}[\cite{jablonski-09}]
  For $n \ge 23,$ the moduli space of indecomposable $n$-dimensional
  two-step nilpotent Lie algebras contains a continuous family of
  nonsoliton Lie algebras.
\end{thm}

There are two key issues involved in the results of Jablonski and
Will.  First of all, nonsoliton nilpotent Lie algebras
are quite rare, and two-step nilpotent Lie algebras are not classified
in dimensions 10 and higher.  Therefore, finding examples of
nonsoliton nilpotent Lie algebras, or curves of them, requires a thorough
understanding of the structure of nilpotent Lie algebras and how that
structure relates to the existence of a soliton inner product.  Second,
in contrast to the semisimple case, there are few fine algebraic
invariants that allow one to distinguish nonisomorphic nilpotent Lie
algebras, so it is a significant task to show that the curves of
nonsoliton nilpotent Lie algebras are mutually nonisomorphic.  Will
used the Pfaffian defined by Scheuneman in \cite{scheuneman-67} to
distinguish the Lie algebras in her families, while Jablonski used
geometric invariant theory.

Our main result is that the moduli space of indecomposable $\boldN$-graded
$n$-dimensional nonsoliton nilpotent Lie algebras is not discrete if
$n \ge 8:$
\begin{thm}\label{main thm}  
  The moduli space of indecomposable $n$-dimensional nonsoliton
  $\boldN$-graded nilpotent Lie algebras contains a one-parameter
  family of nonsolitons if $n \ge 8.$
\end{thm}

As a corollary we can say exactly when the nonsoliton Lie algebras 
are isolated in the moduli space:
\begin{cor}  The set of nonsoliton Lie algebras is discrete in the 
moduli space of indecomposable $n$-dimensional nonsoliton
  $\boldN$-graded nilpotent Lie algebras if and only if $n \le 7.$
\end{cor}

The families of nilpotent Lie algebras that we construct to prove the
theorem are the first examples of continuous families of three-step
nonsoliton nilpotent Lie algebras.
It would be interesting to refine the result in Theorem \ref{main thm} by
specializing it to the two-step case,  determining in which
dimensions nonsolitons are discrete in the moduli space of two-step
nilpotent Lie algebras.

This manuscript is organized as follows.  In Section
\ref{preliminaries}, we review necessary background material related
to nilpotent Lie algebras, inner products on Lie algebras, soliton
inner products, and Nikolayevsky (pre-Einstein) derivations of nilpotent Lie algebras.  In
Section \ref{8 and 9}, we present two continuous families of
indecomposable 
$\boldN$-graded nilpotent Lie algebras, one in dimension eight and one
in dimension nine, and we prove that the Lie algebras in the families are
nonsoliton.  We also
describe the derivation algebras of the Lie algebras in the families.  In
Section \ref{higher dim}, we use the $8$- and $9$-dimensional examples from Section \ref{8 and
  9} to construct continuous families of indecomposable $\boldN$-graded nilpotent Lie
algebras in dimensions $n \ge 10.$ We find Nikolayevsky derivations
for these Lie algebras, and  we prove that all of the Lie algebras in the families are
nonsoliton.  Last, we  prove that for any dimension
$n \ge 8,$   the Lie algebras in the family are all mutually nonisomorphic.  
In Section \ref{proof of main theorem}, we combine our results from 
Sections \ref{8 and 9} and \ref{higher dim} to prove the main result.  

 \section{Preliminaries}\label{preliminaries}

\subsection{Lie algebras}

The descending central series of a Lie algebra $\frakg$ is defined by $\frakg^{(1)}
= \frakg$ and $\frakg^{(j + 1)} = [\frakg, \frakg^{(j)}]$ for $j > 1.$
The Lie algebra $\frakg$ is {\em nilpotent} if and only if there is an
integer $r$ so that $\frakg^{(r)}$ is trivial.  If $r$ is the smallest
integer so that $\frakg^{(r + 1)}$ is trivial, then $\frakg$ is said
to be $r$-{\em step} nilpotent.  An {\em $r$-step} nilpotent Lie
algebra is said to of {\em type} $(n_1,n_2, \ldots, n_r)$ if $\dim
(\frakg^{(j)}/ \frakg^{(j+1)}) = n_j$ for $j = 1, \ldots, r.$

Let $\Der(\frakg)$ be the derivation algebra of $\frakg.$ The algebra
$\Der(\frakg)$ has Levi decomposition $\fraks \oplus
(\frakt_s \oplus \frakt_c) \oplus \frakn$ where $\fraks$
is the semisimple Levi factor and the solvable radical $\rad
(\Der(\frakg)) = \frakt \oplus \frakn$ is the
direct sum of its nilradical $\frakn$ and a torus $\frakt.$ 
The torus 
further decomposes as the sum $\frakt = \frakt_s \oplus
\frakt_c$ of an $\boldR$-split torus $\frakt_s$ and a
compact torus $\frakt_c.$ The dimension of $\frakt$ is
called the {\em rank} of $\frakg,$ and the dimension of the 
$\boldR$-split torus $\frakt_s$ is
called the {\em real rank} of $\frakg.$

A Lie algebra is {\em indecomposable} if it cannot be written as the
direct sum of two nontrivial ideals.

\subsection{Metric Lie algebras and soliton inner
  products}\label{metric Lie algebras}

A {\em metric Lie algebra} $(\frakg,Q)$ is a Lie algebra $\frakg$
endowed with an inner product $Q.$ Associated to each metric Lie
algebra is a unique homogeneous space $(G,g),$ where $G$ is the
connected Lie group whose Lie algebra is $\frakg,$ and $g$ is the left
invariant metric on $G$ such that that the restriction of $g$ to the
tangent space $T_e G \cong \frakg$ of $G$ at the identity coincides
with $Q.$ The Ricci endomorphism $\Ric$ for the Riemannian manifold
$(G,g),$ when restricted to $T_e G,$ is an endomorphism of $T_e G
\cong \frakg.$ We call $\Ric_e$ the Ricci endomorphism for the metric
Lie algebra $(\frakg,Q)$ and we abuse notation to let $\Ric$ denote
$\Ric_e.$

Let $(\frakn,Q)$  be a metric nilpotent Lie
algebras with associated homogeneous space $(N,g).$
  Then Ricci form  $\ric_e$ for  $(N,g)$ at the identity  is the bilinear form
on $T_e N \cong \frakn$   given by  
\[ \ric(x,y) =
-\frac{1}{2}\sum_{i=1}^n Q( [x,x_i], [y,x_i])+\frac{1}{4} \sum_{i,j
= 1}^n Q( [x_i,x_j], x) Q( [x_i,x_j], y),\]
for $x, y \in \frakn,$ and 
where $\{x_i\}_{i=1}^n$ is an orthonormal basis for $\frakn.$ 
 Then Ricci endomorphism $\Ric_e = \Ric$ at the identity 
is the endomorphism of $T_e N \cong \frakn$  given by
the condition that $Q(\Ric(x),y) = \ric(x,y),$ for all $x, y \in \frakn.$ 

 Let $(\frakn_1,Q_1)$ and $(\frakn_2,Q_2)$ be metric nilpotent Lie
algebras with associated homogeneous spaces $(N_1,g_1)$ and 
$(N_2,g_2)$ respectively. 
A map $\psi : \frakn_1 \to \frakn_2$
naturally induces the map $\overline{\psi} : N_1 \to N_2,$ where
$\overline \psi= \exp \circ \psi \circ \exp^{-1}.$ E.\ Wilson proved
that $\overline{\psi}$ is an isometry if and only if $\psi$ is an
isometric isomorphism (\cite{wilson82}); i.e. $\psi$ is an isomorphism
and $Q_1(x,y)=Q_2(\psi(x),\psi(y))$ for all $x, y \in \frakn_1.$

A metric Lie algebra $(\frakg,Q)$ is called  {\em soliton} if its Ricci endomorphism $\Ric \in \End(\frakg)$
differs from a scalar multiple of the identity map by a derivation;
that is, there exists a $\beta \in \boldR$ called the {\em soliton
  constant} so that $\hat D = \Ric - \beta \Id$ is a derivation.  In
the case that the  Lie algebra is nilpotent, we
call the inner product a {\em nilsoliton} inner product, we call
$\beta$ the {\em nilsoliton constant} and we call the derivation $\hat
D$ the {\em nilsoliton derivation}.

Let $\frakg$ be a nonabelian Lie algebra with basis $\calB.$   Let $\Lambda$
index the set of nonzero structure constants $\alpha_{ij}^k$ relative
to $\calB,$ modulo skew-symmetry: 
\[ \Lambda = \{ (i,j,k) \, : \, \alpha_{ij}^k \ne 0, i < j \}. \] To
each triple $(i,j,k) \in \Lambda,$ we associate the {\em root vector}
$\bfy_{ij}^k = \bfe_i + \bfe_j - \bfe_k,$ where $\{\bfe_i\}_{i=1}^m$
is the standard basis for $\boldR^m.$ Let $\bfy_1, \bfy_2, \ldots,
\bfy_m$ be an enumeration of the root vectors $\bfy_{ij}^k, (i,j,k)
\in \Lambda,$ using some fixed ordering of $\boldN^3.$  Our
convention is to order $\boldN^3$ so that for $i_1,j_1,k_1,i_2, j_2,
k_2 \in \boldN,$
\begin{itemize}
\item{$(i_1,j_1,k_1) < (i_2,j_2,k_2)$ if $k_1 < k_2$}
\item{$(i_1,j_1,k_1) < (i_2,j_2,k_1)$ if $i_1 < i_2$}
\item{$(i_1,j_1,k_1) < (i_1,j_2,k_1)$ if $j_1 < j_2.$}
\end{itemize}
The {\em
  Gram matrix} for $\frakg$ with respect to $\calB$ is the matrix $U =
(u_{ij})$ whose entries are the inner products of the root vectors:
$u_{ij} = \bfy_i \cdot \bfy_j.$ We will use the following theorem of
Nikolayevsky to show that the nilpotent Lie algebras in our families
do not admit soliton inner products.

\begin{thm}[Theorem 3, \cite{nikolayevsky-preEinstein}]\label{Uv}  
Let $\frakn$ be a nonabelian nilpotent Lie algebra with basis $\calB.$   Suppose
that the Gram matrix $U$ for $\frakn$ with respect to $\calB$ has no
entries of $2.$  Then $\frakn$ admits a soliton inner product if
and only if the matrix equation $U \bfv = [1]$ has a solution $\bfv$
with all positive entries.  
\end{thm}
Note that we have replaced Nikolayevsky's hypothesis that the basis is
a ``nice basis''
with an equivalent hypothesis on the Gram matrix $U$ (which depends
only on the index set $\Lambda$).   

\subsection{Nikolayevsky derivations}

A derivation $D^N$ of a Lie algebra $\frakg$ is called a {\em
  Nikolayevsky derivation} if it is semisimple with real eigenvalues
and
\begin{equation}\label{defn N der}
\trace(D^N \circ F) = \trace (F) 
\end{equation}
for all $F \in \Der(\frakg).$ Nikolayevsky defined such derivations
and showed that they are unique up to automorphism.  He called them 
{\em pre-Einstein derivations} because when the underlying Lie algebra
is nilpotent, they are natural
generalizations of the nilsoliton derivation used to define an Einstein
solvable extension.   Because
they are purely algebraic objects of broader use,  we prefer to
call such a derivation a {\em Nikolayevsky derivation.}  
He also showed
that if $\frakg$ admits a soliton inner product, then the nilsoliton
derivation is a scalar multiple of the Nikolayevsky derivation
(\cite{nikolayevsky-preEinstein}).  It follows from the proof of
Theorem 1.1(a) of \cite{nikolayevsky-preEinstein} that $D^N$ is a
Nikolayevsky derivation if and only if the condition in Equation
\eqref{defn N der} holds for all $F$ in an $\boldR$-split torus $\frakt^s$
containing $D^N.$ Thus, it is elementary to find the Nikolayevsky
derivations of a Lie algebra with real rank one:

\begin{prop}\label{rank 1 N der}[\cite{nikolayevsky-preEinstein}]
  Let $\frakg$ be a nilpotent Lie algebra with real rank
  one.  Let $D$ be a nontrivial 
semisimple derivation with real eigenvalues.  Then 
\[ D^N = \frac{\trace(D)}{\trace(D^2)} \, D \] is the unique
Nikolayevsky derivation for $\frakg.$
\end{prop}

\subsection{Moduli spaces of soliton and nonsoliton nilpotent Lie
  algebras}\label{moduli spaces}

We need to describe the structure of the moduli space of nilpotent Lie
algebras of dimension $n.$ Let $\boldR^n$ be a real vector space of
dimension $n$ with basis $\calB = \{x_i\}_{i=1}^n.$ Suppose that
$\boldR^n$ is endowed with a Lie bracket that defines a nilpotent Lie
algebra structure on $\boldR^n.$ The Lie bracket is equivalent to a
skew-symmetric vector-valued bilinear map $\mu$ in the vector space $V = \wedge^2
(\boldR^n)^\ast \otimes \boldR^n.$ The Jacobi identity and the
nilpotency condition are polynomial constraints on the coefficients of
$\mu$ in $V$ with respect to the basis $\{ x_i^\ast \wedge x_j^\ast
\otimes x_k \, : \, 1 \le i < j \le n, 1 \le k \le n \},$ so we may identify each $\mu$ with an
element $(\mu, \calB)$ of an affine subvariety $X$ of $V.$ We let
$\frakn_\mu$ denote the corresponding nilpotent Lie algebra.  The
general linear group $GL_n(\boldR)$ acts on $X$ by change of basis:
for $\mu \in X,$ the element $g \cdot \mu$ of $X$ is defined by
\[ (g \cdot \mu) (x,y) = g \mu(g^{-1}x,g^{-1} y), \] for $x, y \in
\boldR^n.$ Two elements $\mu$ and $\nu$ of $X$ define isomorphic Lie
algebras $\frakn_\mu$ and $\frakn_\nu$ if and only if $\mu$ and $\nu$
are in the same $GL_n(\boldR)$ orbit.  The quotient $\calN_n$ of $X$
by this action is the moduli space of $n$-dimensional nilpotent Lie
algebras.  
The equivalence class of $\mu$ is denoted by
$\overline{\mu}.$ We endow $\calN_n$ with the quotient topology.  

The properties of whether a Lie algebra $\frakn_\mu$ for $\mu \in X$
is $\boldN$-graded, and whether it is indecomposable are both
invariant under the $GL_n(\boldR)$ action.  Hence we may define the
moduli space  $\widetilde \calN_n$  of $\boldN$-graded, indecomposable nilpotent Lie algebras
to be the set of elements $\overline \mu$ of $\calN_n$ so that
$\frakn_\mu$ is $\boldN$-graded and indecomposable.  We use the
subspace topology for $\widetilde \calN_n.$   The property of whether
or not a Lie algebra $\frakn_\mu$ admits a soliton inner product is
also invariant under the  $GL_n(\boldR)$ action, so we may define the 
set $\text{Nonsol}(n) \subseteq \widetilde \calN_n$ to be the set of all
$\overline{\mu}$ in $\widetilde \calN_n$ such that $\frakn_\mu$ does
not admit a soliton inner product. 

\section{Continuous families of nonsoliton nilpotent Lie algebras in
  dimensions 8 and 9}\label{8 and 9}

\subsection{A curve of nonsoliton nilpotent Lie algebras in dimension 8}
\begin{defn}\label{dim 8}  Let $s \in \boldR,$  and let   $\calB =
  \{x_i\}_{i=1}^8$ be a fixed basis for $\boldR^8.$ 
 Define $\frakn_s$ to be the 
  nilpotent Lie algebra with underlying vector space $\boldR^8$ whose 
  Lie algebra structure is determined by the bracket relations 
\begin{align}
&[x_2,x_3]= e^{-s} x_4 &
&[x_1,x_3]= e^s x_5  &
&[x_1,x_2]=x_6  \notag \\  
&[x_2,x_6]= e^s x_7  &
&[x_3,x_4]= e^{-s} x_7 &
&[x_1,x_6]= e^{-s} x_8 \label{m8}\\ 
&[x_2,x_4]= x_8  &
&[x_3,x_5]= e^{s} x_8 & 
& \strut   .  \notag 
\end{align}

The Jacobi Identity may be checked by confirming that there are no distinct
choices of $i, j$ and $k$ so that $[x_i,[x_j,x_k]]$ is nonzero.
(That there are no such choices may also be deduced from Theorem 7 of \cite{payne-09b}).

It is not hard to verify that for all $s,$ the Lie algebra $\frakn_s$
is three-step nilpotent of type $(3,3,2)$ and is indecomposable.  For
all $s,$ we may write the vector space $\frakn_s$ as the direct sum
$\frakn_s = V_1 \oplus V_2 \oplus V_3,$ where $V_1, V_2$ and $V_3$ are
the three steps
\begin{align*}
V_1 &= \myspan \{ x_1, x_2, x_3 \} \\
V_2 &= \myspan \{ x_4, x_5, x_6 \} \\
V_3 &= \myspan \{ x_7, x_8 \} .
\end{align*}
Define the derivation $D: \frakn_s \to \frakn_s$ by
\begin{equation*}
  D(x) =  k x,  \quad \text{ if $x \in V_k,$ for $k=1,2,3.$ }
\end{equation*}
The eigenspaces $V_1, V_2, V_3$ for $D$ define an $\boldN$-grading of $\frakn_s.$
\end{defn}

Now we describe the derivation algebra of a typical nilpotent Lie algebra in the family defined
in  Definition \ref{dim 8}.

\begin{prop}\label{der alg 8}  Let $\frakn_s$ be the nilpotent Lie algebra
  as defined in Definition \ref{dim 8}, for any fixed $s$ in $\boldR.$
  Then the derivation algebra $\Der(\frakn_s)$ of $\frakn_s$ is a 
$16$-dimensional solvable algebra with real
  rank one.  The derivation algebra decomposes as  $\Der(\frakn_s) = \boldR
  \, D + \frakm,$ where $D$ is the derivation $D$ defined above and
  $\frakm$ is the nilradical.  The derivation $D^N = \frac{5}{11}D$ is
a  Nikolayevsky derivation of $\frakn_s.$
\end{prop}

\begin{proof} The derivation algebra of $\frakn_s$ is a subspace of
  $\End(\frakn_s).$ The subspace may be described by a system of $8^3$
  linear equations in $8^2$ unknowns, where the coefficients of the
  linear equations depend on the structure constants for $\frakn_s.$
  (See Section 1.9 of \cite{degraaf-book}.)  The structure constants
  for $\frakn_s$ depend on the parameter $s.$ Using Matlab to solve
  this system symbolically, we found that for any $s,$ the solution
  space to the linear system is $16$-dimensional and is spanned by the
  derivation $D$ and $15$ nilpotent derivations that span the
  nilpotent subalgebra $\frakm.$ Hence any semisimple derivation of
  $\frakn_s$ is a scalar multiple of $D$ and the real rank of
  $\frakn_s$ is one.  By Proposition \ref{rank 1 N der}, $D^N =
  \frac{5}{11}D$ is a Nikolayevsky derivation of $\frakn_s.$
\end{proof}

Now we show that none of the nilpotent Lie algebras in the family defined
in  Definition \ref{dim 8} are soliton.

\begin{thm}\label{dim 8 theorem} Suppose that $\frakn_s$ is a nilpotent Lie algebra
as  defined in Definition \ref{dim 8}.  Then $\frakn_s$ does not admit
a soliton inner product. 
\end{thm}

\begin{proof}
Fix $s$ in $\boldR$ and suppose that $\frakn_s$ admits a soliton inner
product $Q.$   With respect to the basis $\calB,$ the Gram matrix $U$ 
is  
\begin{equation}\label{U 8 def}
 U = \begin{bmatrix}  
3 &1 &1 & 1 & 0 & 0 & 0 &1 \\  
1 &3  &1 & 0 & 1 & 1 & 0& 0\\  
1 &1 &3 & 0 & 0 & 0  & 1& 0 \\  
1 & 0 & 0 & 3& 1 &1  & 1 & 0 \\  
 0 & 1 & 0 & 1& 3 & 0 &1 & 1\\  
 0 & 1 & 0 & 1& 0& 3& 1& 1\\  
 0 &0  &1 & 1& 1& 1& 3 & 1\\  
 1 & 0 & 0 & 0& 1& 1&1 & 3 
\end{bmatrix} .\end{equation}

The solution space to the linear equation 
$U \bfv  = [1]$ is
$\{ \bfv_0 +  t \bfv_1 \, : \, t \in \boldR \},$
where 
\begin{align*}
 \bfv_0 &= \smallfrac{1}{11} (1,1,3,2,2,2,0,2)^T, \quad \text{and} \\
\bfv_1 &= (-1,1,0,1,-1,-1,0,1)^T.
\end{align*}
For all such solutions $\bfv = (v_i),$  $v_7 = 0.$ By Theorem
\ref{Uv},
$\frakn_s$ does not admit a soliton inner product.  
\end{proof}

We will show in Theorem \ref{nonisomorphic n} that no two nilpotent Lie algebras in the family defined
in  Definition \ref{dim 8} are isomorphic.

\subsection{A curve of nonsoliton nilpotent Lie algebras in dimension 9}

Now we define a one-parameter family of $9$-dimensional nilpotent Lie
algebras, similar to the family of $8$-dimensional nilpotent Lie
algebras defined in the previous section.
\begin{defn}\label{dim 9}  Let $\calB = \{x_i\}_{i=1}^9$
be a basis for $\boldR^9.$ Let $s$ be a real number.
 Let $\frakn_{s}$ be the  Lie algebra with underlying vector space $\boldR^9$ whose Lie
  algebra structure is determined by the Lie bracket relations 
\begin{align*}
&[x_2,x_3]= e^{4s} x_4 &
&[x_1,x_3]= e^{-3s}  x_5  &
&[x_1,x_2]= e^{-s}x_6  \\ 
& [x_2,x_6]= e^{-4s}  x_7  &
&[x_3,x_4]= e^{4s} x_7 &
&[x_1,x_6]= e^{4s} x_8 \\ 
&[x_2,x_4]= x_8  &
&[x_3,x_5]= e^{-4s} x_8 &
& [x_3,x_6] = e^{-s}x_9 \\
& & 
& [x_2,x_5]  = e^s x_9. 
& 
\end{align*}
The Jacobi Identity may be confirmed by direct computation, noting that
the only time $[[x_i,x_j], x_k]$ is nontrivial  for distinct $i,j,k$ is when $\{i,j,k\} = \{1,2,3\}.$
(The latter fact follows from Theorem 7 of \cite{payne-09b}.)

Each Lie algebra $\frakn_s$ is three-step nilpotent of type $(3,3,3).$
 We may write 
the vector space $\boldR^9 \cong \frakn_s$ as the direct sum $\boldR^9 = V_1 \oplus
V_2 \oplus V_3,$ where $V_1, V_2$ and $V_3$ are the three steps for
any $\frakn_s:$
\begin{align*}
V_1 &= \myspan \{ x_1, x_2, x_3 \} \\
V_2 &= \myspan \{ x_4, x_5, x_6 \} \\
V_3 &= \myspan \{ x_7, x_8, x_9  \} .
\end{align*}
A derivation $D$ of $\frakn_s$ is defined by
\begin{equation*}
  D(x) =  k x,  \quad \text{ if $x \in V_k,$ for $k=1,2,3.$ }
\end{equation*}
The eigenspaces $V_1, V_2, V_3$ for $D$ define an $\boldN$-grading of $\frakn_s.$
\end{defn}

\begin{prop}\label{der alg 9}  
Let $s \in \boldR,$ and let $\frakn_s$ be as defined in Definition
\ref{dim 9}.
 The derivation algebra $\Der(\frakn_s) = \boldR \, D + \frakm$ of
$\frakn_s,$ is
  $19$-dimensional and solvable, with 
$18$-dimensional nilradical $\frakm.$  The real rank of
$\frakn_s$ is one, and $D^N = \frac{9}{14}D$ is a Nikolayevsky
derivation of $\frakn_s.$  
\end{prop}

The proof of the proposition is analogous to that of Proposition
\ref{der alg 8}, so we do not include it.  
 Now we show that none of the Lie algebras  defined in Definition
 \ref{dim 9} are soliton. 
\begin{thm}\label{dim 9 theorem} Let $s \in \boldR,$  and let $\frakn_s$ be a nilpotent Lie
  algebra as   defined in Definition \ref{dim 9}. Then $\frakn_s$
  does not admit a soliton inner product. 
\end{thm}

\begin{proof}  The proof is the same as that of Theorem \ref{dim 8 theorem}, except that 
\begin{equation}\label{U9 def} U = 
\begin{bmatrix}
3  & 1 & 1& 1& 0  & 0  & 0 &1 &1 &1 \\
1 & 3 & 1& 0 & 1& 1 & 0& 0& 1 & -1\\
1 & 1 & 3& 0& 0& 0& 1& 0 & -1 & 1\\
1 & 0 & 0& 3& 1& 1& 1& 0& 1 & 1\\
0 & 1 & 0& 1& 3& 0& 1& 1& 1 & 0\\
0 & 1 & 0& 1& 0& 3& 1& 1& 1 & 0\\
0 & 0 & 1& 1& 1& 1& 3& 1& 0 & 1\\
1 & 0 & 0& 0& 1& 1& 1& 3& 1 & 1\\
1 & 1 & -1& 1& 1& 1& 0& 1& 3 & 1\\
1 & -1 & 1& 1& 0& 0& 1& 1& 1 & 3\\
\end{bmatrix}
 \end{equation}
and the solution space to $U \bfv = [1]$ is  
$\{ \bfv_0 + s \bfv_1 + t\bfv_2 \, : \, s, t \in \boldR \},$
where
\begin{align*} \bfv_0 &=
  \smallfrac{1}{161}(5,25,39,18,28,28,0,18,16,30)^T. \\
  \bfv_1 & = (-1,1,0,1,-1,-1,0,1,0,0)^T \\
  \bfv_2 & = (0,1,-1,0,0,0,0,0,-1,1)^T 
\end{align*}
All vectors $\bfv = (v_i)$ in the solution space have $v_7 = 0.$
\end{proof}

We will show in Theorem \ref{nonisomorphic n}
of Section \ref{higher dim}  that if $\frakn_s$ and $\frakn_t$ are nilpotent Lie algebras
  in the family defined in Definition \ref{dim 9}, $\frakn_s$ and
  $\frakn_t$ are isomorphic if and only if $s=t.$   

\section{Constructions of curves of nonsoliton nilpotent Lie algebras
  in higher dimensions}\label{higher dim}

\subsection{Examples in dimensions $n \ge 10$}

We describe how we construct the higher-dimensional examples from the
$8$- and $9$- dimensional ones already defined.
\begin{defn}\label{dim n+1}
In each even dimension $8+2k,$ where $k \in \boldN,$ a family of Lie algebras  $\frakn_s^{8 + 2k}$ is
defined as follows.   For each $s \in \boldR,$
let $\frakn_s$ be the $8$-dimensional  nilpotent Lie algebra  defined in Definition
\ref{dim 8}. For $k \ge 1,$
the $(8+2k)$-dimensional nilpotent Lie algebra $\frakn_s^{8 + 2k}$ is represented with
respect to a basis $\{ x_i \}_{i=1}^8 \cup \{ y_i \}_{i=1}^{2k}$
so that the bracket relations defining the Lie algebra structure are
those  listed for $\frakn_s$  in Equation \eqref{m8} of Definition \ref{dim 8}, along with
the $k$ additional generating bracket relations  
\begin{equation*}\label{new brackets 1} 
 [y_i, y_{2k-i}] = x_8, \quad \text{for $i=1, \ldots,
   k.$}\end{equation*}
If necessary, we may let $x_{m+i} = y_i$ for $i=1,\ldots, 2k$ to
define an ordering the basis $\calB = \{x_i\}_{i=1}^m \cup \{y_i\}_{2k} =
\{x_i\}_{i=1}^{m + 2k}$ in accord with the subscripts on the $x_i$'s.    
When $k=0,$ we let  $\frakn_s^{8 + 2k} = \frakn_s.$  

In odd dimensions $9+2k,$ the Lie algebras  $\frakn_s^{9 + 2k}$ are
defined similarly.  For any $s \in \boldR,$  
let $\frakn_s$ be  the $9$-dimensional  nilpotent Lie algebra  defined in Definition
\ref{dim 8}.  For $k \ge 1,$
the $(9+2k)$-dimensional nilpotent Lie algebra $\frakn_s^{9 + 2k}$ is represented with
respect to the basis $\{ x_i \}_{i=1}^9 \cup \{ y_i \}_{i=1}^{2k}$
with the bracket relations for $\frakn_s$  in Definition \ref{dim 9}, along with
the additional  bracket relations 
\begin{equation*}\label{new brackets 2} 
 [y_i, y_{2k-i}] = x_9, \quad \text{for $i=1, \ldots,
   k.$}\end{equation*}
When $k=0,$ we let  $\frakn_s^{9 + 2k} = \frakn_s.$  

One may confirm without too much effort that the Jacobi Identity holds for all of these 
product structures.  For any Lie algebra $\frakn_s^{8 + 2k}$  or
$\frakn_s^{9 + 2k},$ with $k \in \boldN,$ the only time a
double bracket $[[x_i,x_j], x_k]$  vanishes for distinct $i,j,k$ is when $\{i,j,k\} =
\{1,2,3\}.$ (This follows from Theorem 7 of \cite{payne-09b}.)
One may also verify that for all $s \in \boldR$ and $k \ge 0,$  the  Lie algebra $\frakn_s^{8 + 2k}$ is  
three-step nilpotent of type $(2k + 3, 3 ,2),$ and  the  Lie algebra $\frakn_s^{9 + 2k}$ is  
three-step nilpotent of type $(2k + 3, 3 ,3).$

Now fix $\frakn_s^{m + 2k},$ where $m = 8$ or $9,$ and $k \in \boldN.$ 
Define the subspaces $V_2, V_3, V_4$ and $V_6$ by 
\begin{align}\label{def grading}
\begin{split} 
V_2 &= \myspan \{ x_1, x_2, x_3\} \\
V_3 &= \myspan \{ y_1, \ldots, y_{2k}\} \\
V_4 &= \myspan \{ x_4, x_5, x_6 \} \\
V_6 &= \myspan \{ x_7, \ldots, x_m \}. 
\end{split}
\end{align}
When $k=0,$ we let $V_2, V_4,$ and $V_6$ be as defined above, and we
let $V_3 = \{0\}.$
Then $\frakn_s^{m + 2k} = V_2 \oplus V_3 \oplus V_4 \oplus V_6$ for
all $k \ge 0.$

For $k \ge 0,$ define the derivation $D: \frakn_s^{m + 2k}\to \frakn_s^{m + 2k}$ by 
\begin{equation}\label{D-gen}
  D(x) =  k x,  \quad \text{ if $x \in V_k,$ for $k=2,3, 4, 6.$ }
\end{equation}
Because $D$ is a derivation, for $m=8$ or $9,$ and $k \ge 0,$ the 
eigenspaces $V_i$ define an $\boldN$-grading of the Lie algebra $\frakn_s^{m + 2k}:$
$[V_i,V_j] \subseteq V_{i + j},$ where $V_k = 0$ when $k \not \in \{2,3,4,6\}.$ 
\end{defn}

We will need the following lemma later.

\begin{lemma}\label{der lemma}  Let $\frakg$ be a Lie algebra, and let $\fraki$
and $\frakj$ be ideals in $\frakg$ such that 
$\frakg$ is the sum (not necessarily a direct sum) of $\fraki$ and
    $\frakj,$ and $[\fraki,\frakj] = 0.$

Let $\pi:\frakg \to \frakg$
denote a projection map from $\frakg$ to $\fraki;$ i.e., an
endomorphism such that
$\pi|_{\fraki} = Id|_{\fraki}$ and $\pi(\frakg) = \fraki$. 
Let $D$ be a derivation of $\frakg.$ 
 Then the restriction of $\pi_1 \circ D$ 
to $\fraki$ is a derivation of $\fraki.$ 

Conversely, if $D_1: \fraki \to \fraki$ is a derivation of
$\fraki,$ $D_2: \frakj \to \frakj$ is a derivation of
$\frakj,$  and  $D_1(z) = D_2(z)$ for all $z \in \fraki \cap
\frakj,$ then 
\[  D(z) = 
\begin{cases}  
D_1(z) & z \in \fraki \\
D_2(z) & z \in \frakj
\end{cases} \]
is a derivation of $\frakg.$

The solvable radical of $\Der(\frakg)$ contains the solvable
radical of $\Der(\fraki).$ 
\end{lemma}

\begin{proof} 
There exists a basis $\{x_i\}_{i=1}^{m} \cup \{ y_j \}_{j=1}^{d}$ for $\frakg$ such that 
\[
\fraki  = \myspan \{x_i\}_{i=1}^{m}  \qquad \text{and}  \qquad
 \frakj < \myspan \{ y_j \}_{j=1}^{d}+ \frakz, 
\]
and  the projection $\pi: \frakg \to
\fraki$  from $\frakg$ to $\fraki$ is given by 
\begin{align*} 
\pi(x_i) &= x_i,  \quad \text{for $i=1, \ldots, m,$ and} \\
  \pi(y_j) &= 0 \quad \text{for   $j=1, \ldots, d.$}\end{align*}

Because $D$ is a derivation, for all $i,j = 1, \ldots, m,$
\begin{equation}\label{def derivation}  D([x_i,x_j]) = [D(x_i), x_j] + [x_i, D(x_j)]. \end{equation}
As $[y_l, x_j] = 0$ for all $l = 1, \ldots, d$ and $j=1, \ldots, m,$ we get 
\begin{equation}\label{der} D([x_i, x_j])  = [(\pi_1 \circ D)(x_i) ,x_j] + [x_i, (\pi_1 \circ
D)(x_j)] \end{equation}
for all  $i,j = 1, \ldots, m.$ The vectors on the right side of
Equation \eqref{def derivation} are in $\fraki,$ hence
the vector on the left side is also in $\fraki,$ so we have
\[ (\pi \circ D)([x_i, x_j])  = [(\pi \circ D)(x_i) ,x_j] + [x_i, (\pi \circ
D)(x_j)] \]
for all vectors $x_i$ and $x_j$ in the basis $\{x_i\}_{i=1}^m$ of  $\fraki.$ Hence,  the restriction of $\pi \circ
D$ to $\fraki$ is a derivation of $\fraki.$

To prove the converse, note that Equation \eqref{der} holds for
all $x_i \in \fraki$ and $x_j \in \frakj$ due to the fact that $[\fraki,\frakj] =
0,$ while it holds if either both  $x_i$ and $x_j$ are in $\fraki$
or both $x_i$ and $x_j$ are in $\frakj$ because $D_1$ and $D_2$ are
derivations of $\fraki$ and $\frakj$ respectively.  
\end{proof}

The next proposition describes the Nikolayevsky derivations of the Lie
algebras defined in Definition \ref{dim n+1}.
\begin{prop}\label{der alg- n}
  Let $m = 8$ or $9,$ and let $k \ge 0.$ For any $s \in \boldR,$ let $\frakn_s^{m + 2k}$
 an $(m + 2k)$-dimensional nilpotent Lie
  algebra as defined in Definition \ref{dim n+1}. 
The derivation $D^N = \frac{m + k -3}{6m + 3k -26} \, D,$ where $D$ is  defined in
Equation \eqref{D-gen}, is a Nikolayevsky derivation of $\frakn_s^{m +
  2k}.$  
\end{prop}

\begin{proof}    Fix $s,$ and let $\frakn_s^{m + 2k}$ be 
 an $(m + 2k)$-dimensional nilpotent Lie
  algebra as defined in Definition \ref{dim n+1}.  
Then the subspace  $\frakm =\myspan
\{x_i\}_{i=1}^m$ is an ideal.  If $m=8,$ the ideal $\frakm$ is isomorphic to the Lie algebra
$\frakn_s$
defined in Definition \ref{dim 8}, and if $m=9,$ then  $\frakm$ is isomorphic to the Lie algebra
$\frakn_s$
defined  in Definition  \ref{dim 9}.  
Let $\frakh =  \myspan (\{x_m\}  \cup \{y_j\}_{j=1}^{2k}).$   The
subspace $\frakh$ is an ideal isomorphic to the Heisenberg algebra of dimension $2k+1.$

Define the derivation $D_0$ of $\frakm$ by 
\begin{equation}\label{defD0}  D_0(w) =
\begin{cases}
   2\lambda w & w \in V_2 = \myspan \{x_i\}_{i=1}^3  \\
3 \lambda w & w \in V_3 = \myspan \{y_j\}_{j=1}^{2k} \\
4\lambda w &  w \in V_4 = \myspan \{x_i\}_{i=4}^6\\
6\lambda w & w \in V_6 = \myspan \{x_i\}_{i=7}^m
\end{cases},
\end{equation}
where $\lambda =\frac{m + k -3}{6m + 3k -26} .$


We want to show that $D_0$ is a  Nikolayevsky derivation by
showing that 
\[ \trace (D_0  \circ F) = \trace(F)\]
for all $F$ in $\Der(\frakn_s^{m + 2k}).$

Let $\pi_\frakm: \frakn_s^{m + 2k} \to \frakm$ and  $\pi_\frakh:
\frakn_s^{m + 2k} \to \frakh$ be the projection maps defined by the
basis $\myspan \{x_i\}_{i=1}^m \cup \{ y_j\}_{j=1}^{2k}.$
By Lemma \ref{der lemma}, the restriction $\pi_\frakm \circ F|_{\frakm}$ of $F$ to $\frakm$  is a derivation of $\frakm.$   By Proposition
\ref{der alg 8}  (if $m = 8$) or Proposition \ref{der alg 9} (if $m = 9$), 
 the restriction of $F$ to $\frakm$ is a nonzero scalar multiple of  the restriction of 
$D_0$ to $\frakm.$   
 Similarly, the  restriction $\pi_\frakh \circ F|_{\frakh}$ of $F$ to $\frakh$ is a derivation $F_1$ of
 $\frakh.$  
 Thus, we may write $F$ as  the sum of derivations $F = \hat F + (F -
 \hat F),$
where
\begin{align*}
   \hat F(w) &=
\begin{cases} \pi_\frakm \circ F|_{\frakm}(w) & w \in \frakm = \myspan \{ x_1, \ldots, x_{m} \}\\
\pi_\frakh \circ F|_{\frakh} (w) & w \in  \myspan \{ y_j \}_{j=1}^{2k}, 
\end{cases} \\ 
&= \begin{cases} 
cD_0(w)  & w \in \frakm= \myspan \{ x_1, \ldots, x_{m} \}\\
 F_1(w)   & w \in \myspan \{ y_j \}_{j=1}^{2k}, 
\end{cases}
\end{align*}
for some   $c \in \boldR,$  and  $F_1 : \frakh \to \frakh$  is a derivation
of the ideal $\frakh.$  

The derivation $\hat F$ fixes the ideals $\frakm$ and
$\frakh,$  and is block diagonal when represented with respect to the
basis.   The restriction of $\hat F$ to $\frakh$ is the derivation $F_1$ of $\frakh.$
The  derivation  $F - \hat F$ maps $\frakm = \myspan \{ x_1, \ldots, x_{m} \}$
into $\myspan \{ y_j \}_{j=1}^{2k}$ and it maps   $\myspan \{ y_j
\}_{j=1}^{2k}$  into $\frakm = \myspan \{ x_1, \ldots, x_{m} \},$
  and when represented with respect to the
basis $\calB$ in block form has 0 blocks along the diagonal.  

Using the definition of $D_0$ we see that $\trace(D_0 \circ (F - \hat F)) = 0;$ hence,
\[ \trace (D_0  \circ \hat F) = \trace(D_0 \circ F).\]
In addition, $\trace \hat F = \trace F.$  Thus, in order to show that 
$\trace (D_0  \circ F) = \trace(F)$ for all derivations $F,$ it suffices to show that
$\trace (D_0  \circ \hat F) = \trace(\hat F)$ for all derivations $F.$


We will compute $\trace(D_0 \circ \hat F)$ and $\trace(\hat F)$ directly. But first we will prove the
following claim:  $k \cdot \trace \hat F|_{\boldR x_m} =   \trace \hat F|_{V_3}.$ 
It is not hard to confirm the basic fact that the derivation $D^N_{\frakh}: \frakh \to \frakh$
defined by 
\[
 D^N_{\frakh}(v) = 
\begin{cases}
\frac{k+1}{k+2} \, v  &  v  \in V_3 = \myspan
 \{ y_l \}_{l=1}^{2k} \\
2 \, \frac{k+1}{k+2}\, v  & v \in \boldR x_m 
\end{cases}
\]
is a Nikolayevsky derivation for the ideal $\frakh \cong \frakh_{2k
  +1}.$   
By the
defining property of the Nikolayevsky derivation $D^N_{\frakh}$ for $\frakh,$
\[
\trace(D^N_{\frakh} \circ \hat F|_{\frakh}) = \trace(\hat F|_{\frakh}), \]
which becomes
\[
\left( \frac{k+1}{k+2} \right) \left(
\trace(\hat F|_{V_3})  + 2 \trace(\hat F|_{\boldR x_m}) \right)=   
\trace(\hat F|_{V_3}) + \trace(\hat F|_{\boldR x_m}).
\]
After simple arithmetic manipulations we get the desired equality
$k \cdot \trace \hat F|_{\boldR x_8} =  \trace \hat F|_{V_3}.$  
Then 
 \begin{equation}\label{trace is}
\trace (\hat F|_{V_3}) =   k \cdot \trace \hat F|_{\boldR x_8} =
k \cdot  \trace (c D_0)|_{\boldR x_8}   = 6c\lambda k.\end{equation}

We return to the computation of $\trace(D_0 \circ \hat F)$ and
$\trace(\hat F),$ using the definitions of $D_0$ and $\hat F,$ finding 
\begin{align*}
\trace(D_0 \circ \hat F)   &=   \trace \left((D_0 \circ \hat  F)|_\frakm\right) +
\trace \left((D_0 \circ \hat F)|_{V_3}\right)\\   &=   
c\trace \left(D_0^2|_\frakm\right) +
\trace \left((3 \lambda \Id_{V_3} \circ \hat F)|_{V_3}\right),
\end{align*}
where $\Id_{V_3}$ denotes the identity map on $V_3.$  Continuing,
using the definition of $D_0$ we get
\begin{align*} 
\trace(D_0 \circ \hat F)   &=   c(3 \cdot 4\lambda^2 + 3 \cdot 16\lambda^2 + (m-6) \cdot 36  \lambda^2)
+ 3 \lambda \trace (\hat F|_{V_3}) \\
&= c \lambda^2 (36m - 156) + 3\lambda \cdot 6 c \lambda k \quad
\text{from Equation \eqref{trace is}}\\ 
&= c \lambda  (36m - 156 + 18k)  \cdot \lambda,\\
&=   6\lambda c  (6m + 3k - 26)   \frac{m + k -3}{6m + 3k -26}  \quad
\text{by definition of $\lambda$} \\
&= 6\lambda c \, (m + k - 3),
\end{align*}
while parallelly, 
\begin{align*}
\trace(\hat F) &=   \trace (\hat F|_\frakm) +
\trace (\hat F|_{V_3})\\ 
&=   3 \cdot 2c\lambda + 3 \cdot 4c\lambda  + (m-6) \cdot 6 c\lambda 
+ \trace (\hat F|_{V_3}) \\
&= (6c\lambda m  -18c\lambda)+ 6c\lambda k \\
&=   6\lambda c \, (m + k - 3).
\end{align*}

Thus, $D_0$ is a Nikolayevsky derivation for $\frakn_s^{m + 2k}$ as claimed.  
\end{proof}

Now we prove a technical lemma about isomorphisms of the algebras which
we have defined.
\begin{lemma}\label{technical lemma}
Let $m = 8$ or $9,$  let $s\in \boldR,$ and let $k \ge 0.$ Let $\frakn_s^{m + 2k}$
  be a $(m + 2k)$-dimensional nilpotent Lie
  algebra as defined in Definition \ref{dim n+1}, and let $V_3, V_4$ and
  $V_6$ be as defined in Equation \eqref{def grading}.   Suppose that $\phi:
  \frakn_s^{m + 2k} \to \frakn_s^{m + 2k}$ is an isomorphism.   
Then  
\begin{enumerate}
\item{If $m=8,$ then $\phi$ maps the subspaces 
$\boldR x_1 \oplus V_3 \oplus V_4 \oplus V_6$ and
$\boldR x_5 \oplus V_3 \oplus V_6$ to themselves. }\label{dim 8 isomorphism}
\item{If $m=9,$ then $\phi$ maps the subspaces 
$\boldR x_1 \oplus  V_4 \oplus V_6$ and
$\boldR x_4 \oplus \boldR x_5 \oplus V_4 \oplus V_6$ to themselves.}\label{dim 9 isomorphism}
\end{enumerate}
\end{lemma}

\begin{proof}
 The two subspaces $\boldR \, x_1 \oplus V_3 \oplus V_4 \oplus V_6$ and 
$\boldR x_4 \oplus \boldR x_5 \oplus V_3 \oplus V_6$ are fixed by isomorphisms because
each may be uniquely characterized by algebraic
properties that are preserved under isomorphisms.  

We claim that when $m=8,$ the subset
\[ S = \{ b_1 x_1 + v \, : \, b_1 \ne 0, v  \in V_3 \oplus 
 V_4 \oplus  V_6 \}\]
is the set of all elements $x$ such that  $\ad_{x}$ has rank 3, where
$\ad_x$ is the adjoint map for either $\frakn_s^{m + 2k}$ or  $\frakn_t^{m + 2k}.$ 

For example,
when $k=1,$ if $x= \sum_{i=1}^8 b_i x_i + c_1 y_1 + c_2 y_2,$ then 
with respect to the usual basis $\calB,$ the adjoint map $\ad_x$
for $\frakn_s^{10}$ is represented by the matrix 
\[  [\ad_x]_\calB = \begin{bmatrix} 
0 & 0 & 0 & 0 & 0 & 0 & 0 & 0 & 0 & 0 \\
0 & 0 & 0 & 0 & 0 & 0 & 0 & 0 & 0 & 0 \\
0 & 0 & 0 & 0 & 0 & 0 & 0 & 0 & 0 & 0 \\
0 & -e^{-s} b_3 & e^{-s}b_2 & 0 & 0 & 0 & 0 & 0 & 0 & 0 \\
-e^sb_3 & 0 & e^s b_1 & 0 & 0 & 0 & 0 & 0 & 0 & 0 \\
-b_2 & b_1 & 0 & 0 & 0 & 0 & 0 & 0 & 0 & 0 \\
0 &-e^s b_6 & -e^{-s} b_4 & e^{-s} b_3 & 0 & e^s b_2 & 0 & 0 & 0 & 0 \\
-e^{-s} b_6 & -b_4 & -e^s b_5 & b_2 & e^s b_3 & e^{-s} b_1 & 0 & 0 & -c_2 & c_1 \\
0 & 0 & 0 & 0 & 0 & 0 & 0 & 0 & 0 & 0 \\
0 & 0 & 0 & 0 & 0 & 0 & 0 & 0 & 0 & 0 
\end{bmatrix}.\]
The rank of the submatrix 
\[ M_1 = \begin{bmatrix} 
0 & -e^{-s} b_3 & e^{-s}b_2 \\
-e^sb_3 & 0 & e^s b_1 \\
-b_2 & b_1 & 0  
\end{bmatrix}
\]
is two if and only if $(b_1,b_2,b_3) \ne (0,0,0),$ and is zero
otherwise.  

Therefore if the rank of $\ad_x$ is
 $3,$ then $(b_1,b_2,b_3) \ne (0,0,0).$  But if $b_2 \ne 0$ or $b_3
 \ne 0,$ then the minor 
\[\begin{bmatrix} 
e^{-s} b_3 & 0 & e^s b_2 \\
b_2 & e^s b_3 & e^{-s} b_1 
\end{bmatrix}\] 
has rank two.  The block form of the matrix
then forces the rank of the larger matrix to be at least four,  a contradiction.
Hence $b_2 = b_3 = 0,$  and  $x \in S.$
Conversely, if $b_1 \ne 0$ and $b_2 = b_3 = 0,$ then rows 5, 6 and
8 form a basis for the row space of the matrix.

Since isomorphisms preserve the rank of $\ad_x,$ $S$ is invariant 
under isomorphisms.
By continuity, an isomorphism fixes the closure of $S,$
$\boldR_1 x_1 \oplus V_3 \oplus V_4 \oplus V_6.$ 

The subspace $\boldR x_5 \oplus V_4 \oplus V_6$ can be characterized
algebraically as the closure of the set of nonzero elements $x$ in the commutator ideal 
$V_4 \oplus V_6$ such that the rank of $\ad_x$ is 1.   This is 
seen by letting $b_1, b_2, b_3, c_1$ and $c_2$ equal  zero in the
matrix representing $\ad_x.$

Thus we have shown that $\boldR x_1 \oplus V_3 \oplus V_4 \oplus V_6$ and
$\boldR x_5 \oplus V_3 \oplus V_6$ are preserved by isomorphisms, 
establishing Part \eqref{dim 8 isomorphism}
of the lemma in the case that $m=10.$
The same arguments apply in higher even dimensions $8 + 2k > 10.$ 

Now suppose the $m=9.$   Let 
\[ S = \{ b_1 x_1 + v \, : \, b_1 \ne 0, v \in V_4 \oplus V_6 \}.  \]
We assert that $S$ is the 
set of all elements $x$ such that  $\ad_{x}$ has rank 3,
and $x$ is not in the centralizer of the
commutator ideal.
The commutator
 ideal is $V_4 \oplus V_6$ and its centralizer 
is 
\[ \frakz(V_4 \oplus V_6) = V_3 \oplus V_4 \oplus V_6.\]

If $x= \sum_{i=1}^9 b_i x_i + c_1 y_1 + c_2 y_2,$
then adjoint map $\ad_x$
for $\frakn_s^{11}$ is represented by the matrix 
\setcounter{MaxMatrixCols}{12} 
\[  \begin{bmatrix} 
0 & 0 & 0 & 0 & 0 & 0 & 0 & 0 & 0 & 0 & 0 \\
0 & 0 & 0 & 0 & 0 & 0 & 0 & 0 & 0 & 0 & 0 \\
0 & 0 & 0 & 0 & 0 & 0 & 0 & 0 & 0 & 0 & 0 \\
0 & -e^{4s} b_3 & e^{4s}b_2 & 0 & 0 & 0 & 0 & 0 & 0 & 0 & 0 \\
-e^{-3s}b_3 & 0 & e^{-3s} b_1 & 0 & 0 & 0 & 0 & 0 & 0 & 0 & 0 \\
-e^{-a}b_2 & e^{-a}b_1 & 0 & 0 & 0 & 0 & 0 & 0 & 0 & 0 & 0 \\
0 &-e^{4s} b_6 & -e^{4s} b_4 & e^{4s} b_3 & 0 & e^{-4s} b_2 & 0 & 0 & 0 & 0 & 0  \\
-e^{4s} b_6 & -b_4 & -e^{-4s} b_5 & b_2 & e^{-4s} b_3 & e^{4s} b_1 & 0
& 0 & 0 &0 & 0  \\
0 & -e^sb_5 & -e^{-s}b_6 & 0 & e^sb_2 & e^{-s} b_3 & 0 & 0 & 0 & -c_2 & c_1 \\
0 & 0 & 0 & 0 & 0 & 0 & 0 & 0 & 0 & 0 & 0 
\end{bmatrix} .\]
The rank of the submatrix 

\[  \begin{bmatrix} 
0 & -e^{4s} b_3 & e^{4s}b_2 \\
-e^{-3s}b_3 & 0 & e^{-3s} b_1 \\
-e^{-a}b_2 & e^{-a}b_1 & 0
\end{bmatrix} \]
is two if and only if $(b_1,b_2,b_3) \ne 0$ and is zero otherwise.  

Assume that $x$ is not in $\frakz(V_4 \oplus V_6),$ and that 
 the rank of $\ad_x$ is three.  
Since $x$ is not in the centralizer of the commutator,  $(b_1,b_2,b_3)
\ne (0,0,0).$ 
If $b_2 \ne 0$ or $b_3 \ne 0,$ then the minor 
\[  \begin{bmatrix} 
e^{4s} b_3 & 0 & e^{-4s} b_2  \\
b_2 & e^{-4s} b_3 & e^{4s} b_1 \\
0 & e^sb_2 & e^{-s} b_3 
\end{bmatrix} .\]
 has rank 2 or more, forcing the larger matrix to have rank greater than
 three, a contradiction.  Therefore $b_2 = b_3 = 0.$
Substituting these values into the larger matrix, we see that  rows 5,
6 and 8 are independent, and row 9 will not being in their span unless $c_1
= c_2 =0.$      Hence, $x \in S.$

Conversely, if $x \in S,$ then $b_1 \ne 0,$  $b_2 = b_3 = 0,$ and $c_1
= c_2 = 0.$   Since $b_1 \ne 0,$ the vector $x$ is not in the
centralizer of the commutator ideal.  After substituting the zero values into the large matrix,
we see that rows 7 and 9 are in the span of the independent rows 5, 6 and 8.  Hence the
 rank of $\ad_x$ is 3.

 Thus we have shown that $S$ can be
characterized as  the 
set of all $x$ such that $\ad_x$ is not in the centralizer of the commutator
ideal and  $\ad_{x}$ has rank 3.  Therefore $\phi(S) = S,$ 
and $\phi$ preserves the closure $\boldR x_1  \oplus V_4 \oplus V_6.$

We have just shown that an isometry $\phi$ preserves $W = \boldR x_1  \oplus
V_4 \oplus V_6.$  Therefore $\phi$ will also map the centralizer  of $W,$
\[ \frakz(W) = \{ x \in \frakn_s^{m + 2k} \, : \, [x,y] = 0 \enskip
\text{for all $y \in \boldR x_1  \oplus V_4 \oplus V_6$} \}\]
of $W$ for $\frakn_s^{m+2k}$ to the centralizer of $\phi(W)=W$
 for $\frakn_t^{m+2k}.$ 
But 
\[ \frakz(W) = \boldR x_4 \oplus\boldR x_5 \oplus V_3 \oplus V_6. \]
Therefore, $\phi$ preserves $\boldR x_4 \oplus \boldR x_5 \oplus V_3
\oplus V_6$  as claimed.

This we have shown that Part \eqref{dim 9 isomorphism} holds in
dimension $11.$  The same arguments apply in odd dimensions $9 + 2k$ greater than 11.
\end{proof}

Now we are ready to show that for fixed $m$ and $k,$ 
no distinct two Lie algebras in the family  $\frakn_s^{m +  2k}, s \in
\boldR$ are isomorphic. 

\begin{thm}\label{nonisomorphic n}
  Let $m = 8$ or $9,$  let $s, t \in \boldR,$ and let $k \ge 0.$ Let $\frakn_s^{m + 2k}$
  and $\frakn_t^{m + 2k}$ be two $(m + 2k)$-dimensional nilpotent Lie
  algebras as defined in Definition \ref{dim n+1}.  Then $\frakn_s^{m
    + 2k}$ and $\frakn_t^{m + 2k}$ are isomorphic if and only if $s =  t.$
\end{thm}

\begin{proof}  
Suppose that $\phi : \frakn_s^{m + 2k}\to \frakn_t^{m + 2k}$ is an
isomorphism.   We view $\frakn_s^{m + 2k} $ and $\frakn_t^{m + 2k} $ as
the same vector space endowed with different Lie brackets.
  Let $V_i$ denote the eigenspace for 
the derivation $D$ with eigenvalue $i$ as in Equation \eqref{def grading}.   
Recall that 
eigenspaces $V_i,$ where $i=2,3,4,6,$  define an $\boldN$-grading of $\frakn_s^{m+2k}:$
$[V_i,V_j] \subseteq V_{i + j},$ where $V_l = 0$ when $l \in \boldN
\setminus \{2,3,4,6\},$  and $V_3 = \{0\}$ when $k=0.$  In particular,
we know that
\begin{gather*}
 [V_3 \oplus V_4 \oplus V_6, \frakn_t^{m +2k}] \subseteq V_6,  \\
 [V_3 \oplus V_4 \oplus V_6, V_4 \oplus V_6] =  \{0\}, \end{gather*}
and that $V_6$ is central. 

By Lemma \ref{technical lemma}, Part \eqref{dim 8 isomorphism}, the isomorphism $\phi$ maps the subspace
$\boldR x_1$ into  the subspace $\boldR x_1 \oplus V_3 \oplus V_4 \oplus V_6.$
Therefore we may write 
\[ \phi(x_1)  = a_{11} x_1  + v_1 \]
for some vector $v_1 \in  V_3 \oplus V_4 \oplus V_6$ and some $a_{11} \in \boldR.$ 
  As $\phi$ is an isomorphism, $a_{11} \ne 0;$ were $a_{11}$ to vanish, $\phi(x_1)$
  would be in the centralizer of the commutator while $x_1$ was not.
We write 
\begin{align}
\begin{split}
\phi(x_2) & = a_{12} x_1 +  a_{22} x_2 + a_{32} x_3  + v_2\\
\phi(x_3) & = a_{13} x_1 +  a_{23} x_2 + a_{33} x_3  +  v_3
\end{split}
\end{align}
for scalars $a_{12}, a_{22}, a_{32}, a_{13}, a_{23}$ and
$a_{33},$ and vectors $v_2$ and $v_3$ in $V_3 \oplus V_4 \oplus V_6.$

To complete the proof, we will consider the cases $m=8$ and $m = 9$ separately. 
First suppose that $m=8.$  
The first defining relation $[x_2,x_3] = e^{-s} x_4$ yields
\begin{align}\label{DR1-n}
\begin{split}
\phi(x_4) &= e^s [\phi(x_2), \phi(x_3)] \\
 &= e^s  [a_{12} x_1 + a_{22} x_2 + a_{32} x_3, a_{13} x_1 + a_{23} x_2 +
 a_{33} x_3]  + v_4 \\
&=  e^{s-t} (
a_{22} a_{33}   -  a_{32} a_{23})  x_4 + e^{s+t} (a_{12} a_{33}  -
a_{32} a_{13} ) x_5   \\
&  \qquad \qquad  + e^s  (a_{12} a_{23}  - a_{22}
a_{13}  )  x_6  + v_4,
\end{split}  
\end{align} 
where $v_4$ is a vector in the center $V_6.$ 

The second relation  $[x_1,x_3] = e^{s} x_5$ gives us
\begin{align}\label{DR2-n}
\begin{split}
\phi(x_5) &= e^{-s} [\phi(x_1), \phi(x_3)] + v_5 \\
 &= e^{-s} [a_{11} x_1, a_{13} x_1 + a_{23} x_2 + a_{33} x_3]  + v_5\\
&=    e^{-s+t}  a_{11}a_{33} x_5 +  e^{-s}  a_{11}a_{23} x_6 + v_5
\end{split}
\end{align} 
for some vector $v_5 \in V_6.$   Hence $\phi(x_5) \in V_4 \oplus V_6.$
By Lemma \ref{technical lemma},  Part \eqref{dim 8 isomorphism}, the subspace $\boldR x_5$
is mapped into $\boldR x_5 \oplus V_3 \oplus
V_6$  by $\phi.$   Hence, the
$x_6$ coefficient in Equation  \eqref{DR2-n} is  zero. 
 From the fact
that $a_{11}
\ne 0,$  we deduce that  $a_{23} = 0.$ Now we have
\begin{equation}\label{DR2b-n}
\phi(x_5) = e^{-s+t}  a_{11}a_{33} x_5 + v_5.
\end{equation}

Next we get
\begin{align}\label{DR3-n}
\begin{split}
\phi(x_6) &= [\phi(x_1),\phi(x_2)]\\ 
 &= [a_{11} x_1, a_{12} x_1 + a_{22} x_2 + a_{32} x_3 ] + v_6\\
&=  e^t a_{11}a_{32}  x_5   + a_{11}a_{22} x_6+ v_6
\end{split}
\end{align}
for some vector $v_6 \in V_6,$  so $\phi(x_6) \in V_4 \oplus V_6.$

The bracket relation $[x_2,x_6] =e^s x_7$ implies that 
\begin{align}  
\label{DR4-n}
\begin{split}
\phi(x_7)    &=  e^{-s}[\phi(x_2),\phi(x_6)]  \\
&=  e^{-s}[a_{12} x_1 +  a_{22} x_2 + a_{32} x_3, e^t a_{11}a_{32}
x_5   + a_{11}a_{22} x_6] \\ 
&= e^{-s+t} a_{11}a_{22}^2 x_7 
+  a_{11}(e^{-s-t} a_{12}a_{22} + e^{-s+2t} a_{32}^2 )x_8.
\end{split}
\end{align}
We use  the  relation  $[x_3,x_4]=e^{-s}x_7,$ substituting $a_{23}=0$
when it occurs, to get
\begin{align}\label{DR5-n} 
\begin{split}
\phi(x_7) &=   e^{s}  [\phi(x_3),\phi(x_4)]   \\ 
&=     e^{s}  [a_{13} x_1   + a_{23}x_2 + a_{33} x_3  +  v_3,e^{s-t} a_{22} a_{33} x_4   \\
&  \qquad  + e^{s+t} (a_{12} a_{33}  - a_{32} a_{13} ) x_5 -  e^s a_{22} a_{13}    x_6  + v_4]  \\ 
&=     e^{2s-2t} a_{33}^2 a_{22}  x_7    \\  
& \quad  + \big( -e^{2s-t} a_{13}^2  a_{22}    +e^{2s+2t}  a_{33}
(a_{12}a_{33} - a_{32} a_{13}) \big) x_8.  
\end{split}
\end{align}

The bracket relation $[x_1,x_6]=e^{-s}x_8$ gives 
\begin{align}\label{DR8-a-n}
\begin{split}
\phi(x_8) &= e^{s} [a_{11}x_1, e^t a_{11}a_{32}  x_5   + a_{11}a_{22}
x_6] \\
 &= e^{s-t} a_{11}^2 a_{22} x_8,  
\end{split}
\end{align}
 and from $[x_3,x_5]=e^s x_8$ we find
\begin{align}\label{DR8-b-n}
\begin{split}
\phi(x_8) &=    e^{-s} [a_{13} x_1 + a_{23}x_2+ a_{33} x_3,
e^{-s+t}  a_{11}a_{33} x_5]\\
&= e^{-2s+2t} a_{33}^2 a_{11}  x_8. 
\end{split}
\end{align} 
We know that $\phi(x_8) \ne 0,$  so
Equations \eqref{DR8-a-n} and \eqref{DR8-b-n}
tell us that $a_{22} \ne 0$ and    $a_{33} \ne 0.$ 

Equating coefficients of $x_7$ in  Equations  \eqref{DR4-n} and
\eqref{DR5-n} gives 
\begin{equation}\label{first expr-n}
 e^{-s+t} a_{11}a_{22} =  
  e^{2s-2t} a_{33}^2  ,\end{equation}
and equating  $x_8$ coefficients in Equations \eqref{DR8-a-n} and
\eqref{DR8-b-n} gives 
\begin{equation}\label{second expr-n}
e^{s-t} a_{11} a_{22}  =   e^{-2s+2t} a_{33}^2.
\end{equation} 

Thus, 
\[  e^{3s - 3t} a_{33}^2 = a_{11}a_{22} =  e^{-3s + 3t}  a_{33}^2\]
from  Equations \eqref{first expr-n} and \eqref{second expr-n}.
Because $a_{33}$ is nonzero,  we may conclude that $s = t$
as desired. 

Now suppose that $m=9.$    
The bracket relation $[x_2,x_3] = e^{4s} x_4$ implies that
\begin{align}\label{x4-n}
\begin{split}
\phi(x_4) =  e^{4t-4s}& (a_{22} a_{33} - a_{32}  a_{23}) x_4 
+ e^{-4s-3t}(a_{12}a_{33}-a_{32}a_{13}) x_5  \\
& + e^{-4s-t}(a_{12}a_{23}-a_{22}a_{13}) x_6 + v_4, 
\end{split}
\end{align}
where $v_4 \in V_6.$
We use $[x_1,x_3] = e^{-3s} x_5$  to obtain that
\begin{equation*}
\phi(x_5) =  
e^{3s-3t}a_{11}a_{33} x_5 
+ e^{3s-t}a_{11}a_{23} x_6 + v_5
\end{equation*}
for some $v_5 \in V_6.$
By Lemma \ref{technical lemma},
we know that $\phi(x_5)$ is in the invariant subspace $\boldR x_5 \oplus\boldR x_5 \oplus V_4 \oplus V_6,$ so the
coefficient of $x_6$
is zero.  Hence $a_{23} = 0.$ 

The bracket relation $[x_1,x_2] = e^{-s}x_6$ yields
\begin{equation}\label{x6-n}
\phi(x_6) =  
e^{s-3t}a_{11}a_{32} x_5 
+ e^{s-t}a_{11}a_{22} x_6 + v_6
\end{equation}
where $v_6 \in V_6,$ 
and  the bracket relation $[x_1,x_6] = e^{4s} x_8$ gives us  
\[ 
\phi(x_8) = e^{-3s+3t} a_{11}^2 a_{22} x_8.\]
Because $\phi(x_8) \ne 0,$ we see that $a_{22}
\ne 0.$ 

The bracket relation $[x_2,x_6] = e^{-4s} x_7$ becomes 
\begin{align}\label{x7-n}
\begin{split}
\phi(x_7) =  
 e^{5s-5t}a_{11}  &a_{22}^2x_7 
+ a_{11}(e^{5s+3t} a_{22}a_{12} + e^{5s-7t} a_{32}^2) x_8 \\
&+ 2 e^{5s-2t}a_{11}a_{22} a_{32} x_9.
\end{split}
\end{align}
The bracket relation $[x_3,x_4]=e^{4s} x_7$ and that $a_{23}=0$ gives
\begin{align}\label{x7-2-n}
\begin{split}
\phi(x_7) & =  e^{-8s+8t}  a_{22}  a_{33}^2 x_7  \\ 
& \quad + \left( -e^{-8s+3t} a_{22}a_{13}^2 + e^{-8s-7t}
  a_{33}(a_{12}a_{33}-a_{32}a_{13})  \right) x_8  \\ 
& \qquad - e^{-8s-2t} a_{22} a_{13} a_{33}  x_9.  
\end{split}
\end{align}

Setting the $x_7$ coefficients from Equations \eqref{x7-n} and
\eqref{x7-2-n} equal,  we get 
\begin{equation}\label{above-n}
  a_{11}  a_{22} =  e^{13t-13s} a_{33}^2. 
\end{equation}

The bracket relation $[x_3,x_5] = e^{-4s} x_8$ and the fact that $a_{23}=0$ give 
\[ \phi(x_8) = e^{7s - 7t}  a_{33}^2 a_{11} x_8.\]
This means that $a_{33} \ne 0.$
 Equating the $x_8$ coefficients in this expression and the 
previous expression for $\phi(x_8),$ we get 
\[  e^{-3s+3t} a_{11}^2 a_{22} = e^{7s - 7t}  a_{33}^2 a_{11}, \]
so 
\[   a_{11}  a_{22}  = e^{10s-10t} a_{33}^2 .\]
This together with Equation \eqref{above-n} gives $e^{13t-13s} =e^{10s - 10t} .$
Hence, $s=t.$ 
\end{proof}

\begin{thm}\label{nonsoliton - n}
Let $m =8$ or $9,$ and let $k \in \boldN.$  Let $\frakn_s^{m + 2k}, s
\in \boldR$
 be the one-parameter family of  nilpotent $(m + 2k)$-dimensional nilpotent Lie
  algebras  defined in Definition \ref{dim n+1}.   None of the Lie
  algebras in the family are soliton. 
\end{thm}

\begin{proof}
 Let $x_{m+i} = y_i$ for
$i=1,\ldots, 2k.$   Then the index set $\Lambda$ with respect to  the
basis $\calB = \{x_i\}_{i=1}^m \cup \{y_j\}_{j=1}^{2k} = \{x_i\}_{i=1}^{m+2k}$ is
$\Lambda_\frakm \cup \Lambda_\frakh,$ where
$\Lambda_\frakm$ is the index set for $\frakm$ with respect to the
basis $\{x_i\}_{i=1}^m,$
and 
\[\Lambda_\frakh = \{ (m+ i, m+ 2k - i, 8) \, : \, i=1, \ldots, k \}.\]
The set $\Lambda$ is ordered as described in Section \ref{metric Lie
  algebras}.  In this ordering, if $(i_1,j_1,k_1) \in \Lambda_\frakm$
and $(i_2,j_2,k_2) \in \Lambda_\frakh,$ then $(i_1,j_1,k_1) < (i_2,j_2,k_2).$

First we consider the case that $m=8.$
Let the family $\frakn_s^{8 + 2k}, s \in \boldR$ of nilpotent Lie algebras
 be as defined in Definition \ref{dim n+1}.
  We will do a proof by contradiction, so we suppose that 
$\frakn_s^{8 + 2k}$ admits a soliton inner product.

  Let $U$ denote the Gram matrix for 
$\frakn_s^{i + 2k}$ with respect to the basis $\calB.$ 
For $a, b \in \boldN,$ let $[0]_{a \times b}$ denote the $a \times b$ matrix with
 all entries zero, and let $[1]_{a \times b}$ denote the $a \times b$
 matrix with all entries one, and let $I_a$ denote the $a \times a$
 identity matrix. 
The matrix $U$ has block form  
\[   U = 
\begin{bmatrix} 
  U_{11} &  U_{12} &  [0]_{5 \times k} \\
  U_{21} & U_{22} & [1]_{3 \times k} \\
  [0]_{k \times 5}  &   [1]_{k \times 3} & 3 I_k 
\end{bmatrix},
\]
where 
\[   U_8= 
\begin{bmatrix} 
U_{11}  & U_{12} \\
U_{21}  & U_{22} 
\end{bmatrix} \]
is the matrix $U$ in Equation \eqref{U 8 def}, 
broken into blocks $U_{11}, U_{12}, U_{21}, U_{22}$ of sizes 
$5 \times 5, 5 \times 3, 3 \times 5,$ and $3 \times 3$ respectively.   

By Theorem \ref{Uv} there exists a solution
$\bfv$  to $U\bfv = [1]$ with all positive entries.  
We may write $\bfv$ as
\[  \bfv = \begin{bmatrix} \bfv_1 \\ \bfv_2  \\ \bfv_3 \end{bmatrix}, \]
 where $\bfv_1 = (v_i)_{i=1}^5$ is $5 \times 1,$ $\bfv_2 =
 (v_i)_{i=6}^8$ is $3 \times 1$ and $\bfv_3 = (v_i)_{i=9}^{8+k}$ is $k \times 1.$ 
Multiplying $U\bfv$ in block form gives
\begin{align*}
  U_{11} \bfv_1 + U_{12} \bfv_2 &= [1]_{5 \times 1} \\
  U_{21} \bfv_1 + U_{22} \bfv_2  + [1]_{3 \times k} \bfv_3 &= [1]_{3
    \times 1}  \\
 [1]_{k \times 3} \bfv_2 + 3 \bfv_3  &= [1]_{k
    \times 1}. 
\end{align*}
Substituting $[1]_{3 \times k} \bfv_3  = ( \sum_{i=9}^{8 + k} v_i)
[1]_{3 \times 1}$ and  $ [1]_{k \times 3} \bfv_2 =(\sum_{i=6}^{8}
v_i)[1]_{k \times 1}$ into the above yields the  equivalent system   
\begin{align}
  U_{11} \bfv_1 + U_{12} \bfv_2 &= [1]_{5 \times 1}  \label{U1}\\
  U_{21} \bfv_1 + U_{22} \bfv_2  &= \left(1 - \sum_{i=9}^{8 + k} v_i
  \right) [1]_{3
    \times 1} \label{U2} \\
3  \bfv_3  &= \left(1 - \sum_{i=6}^{8} v_i
  \right)[1]_{k    \times 1} \label{U3}.
\end{align}

 Equation \eqref{U3} implies that $\bfv_3 = c \, [1]_{k \times 1},$
where
\begin{equation}\label{defn c} c = \frac{1}{3} \left(1 - \sum_{i=6}^{8} v_i
  \right)> 0.\end{equation}
It follows that
\[\sum_{i=9}^{8 + k} v_i = k c = \frac{k}{3} \left(1 - \sum_{i=6}^{8} v_i
  \right).
 \]
Now we bound the coefficient 
\[ a = 1- \sum_{i=9}^{8 + k} v_i\]
of $[1]_{3 \times 1}$ in Equation \eqref{U2}.
The matrices $U_{21}$ and $U_{22}$ are nontrivial and have no negative entries, and
the entries of the vectors $\bfv_1$ and $\bfv_2$ are positive.   Hence,
Equation \eqref{U2} forces $a$ to be positive.

On the other hand,
\begin{align*}
 a = 1- \sum_{i=9}^{8 + k} v_i&= 1 - k c \\ &= 1 - \frac{k}{3} \left(1 - \sum_{i=6}^{8} v_i
  \right) \\
&= \frac{3-k}{3} + \frac{k}{3} \sum_{i=6}^{8} v_i .
\end{align*}
The inequality in Equation \eqref{defn c} implies that $\sum_{i=6}^{8} v_i < 1$ , so 
\[ a < \frac{3-k}{3} + \frac{k}{3}  = 1. \]
Thus,  we have shown that  
$a$ lies in the interval $(0,1).$

Returning to Equations \eqref{U1}-\eqref{U3},
we have that 
\[  \bfx = \begin{bmatrix}  \bfv_1 \\ \bfv_2 \end{bmatrix} \]
is a solution to the matrix equation
\[  
U_8 \, \bfx = 
\begin{bmatrix} 
  U_{11} &  U_{12}  \\
  U_{21} & U_{22}  
\end{bmatrix}  
\,  \bfx
= 
\begin{bmatrix} 
  [1]_{5 \times 1} \\
  [a]_{k \times 1}
\end{bmatrix} .
\]

Now this involves only the matrix $U_8,$ which is given explicitly in
Equation \eqref{U 8 def}.  Solving symbolically using
Matlab,
we find that the general solution to the matrix equation is
$\bfv = \bfv_0 + t \bfv_1,$ where
\[
\bfv_0 =  \smallfrac{1}{33}(7a + 2, 3-6a, 13-4a, 11-11a, 2a+10, 13a-1,
11a-11,0)^T\]
and 
\[ \bfv_1 = (-1,1,0,1,-1,-1,0,1)^T. \]
For all such solutions $\bfv = (v_i),$ the component 
$v_7$ is $\frac{a-1}{3}.$  We assumed that $v_7 > 0,$ so then $a - 1>0,$ a
contradiction to the fact that $a \in (0,1).$

Therefore, for all $s \in \boldR,$ and all $k \in \frakn,$ the Lie
algebra $\frakn_s^{8 + 2k}$ is not soliton, as claimed. 

Now suppose that $m=9.$   Let 
$\frakn_s^{9 + 2k}$ be one of the Lie algebras in the one-parameter family of nilpotent Lie algebras
 defined in Definition \ref{dim n+1}.  Suppose that  $\frakn_s^{9 +
   2k}$ admits a nilsoliton inner product.

 Let $U$ be the Gram matrix for 
$\frakn_s^{9 + 2k}$ with respect to the basis $\{x_i\}_{i=1}^m \cup
\{y_j\}_{j=1}^{2k} = \{x_i\}_{i=1}^{9 + 2k}.$  By examining the index
set $\Lambda,$ we see that 
the matrix $U$ has is of form
\[   U = 
\begin{bmatrix} 
  U_{11} &  U_{12} &  [0]_{8 \times k} \\
  U_{21} & U_{22} & [1]_{2 \times k} \\
  [0]_{8 \times k}  &   [1]_{2 \times k} & 3 I_k \\
\end{bmatrix},
\]
where 
\[   U_9= 
\begin{bmatrix} 
U_{11}  & U_{12} \\
U_{21}  & U_{22} 
\end{bmatrix} \]
is the $10 \times 10$ matrix $U$ in Equation \eqref{U9 def}, and the 
blocks $U_{11}, U_{12}, U_{21}$ and $U_{22}$ are size
$8 \times 8, 8 \times 2, 2 \times 8,$ and $2 \times 2$ respectively. 

By Theorem \ref{Uv}, the matrix equation $U\bfv = [1]$ has a
solution $\bfv$ with all positive entries.  
We write $\bfv$ as
\[  \bfv = \begin{bmatrix} \bfv_1 \\ \bfv_2  \\ \bfv_3 \end{bmatrix}, \]
 where $\bfv_1$ is $8 \times 1,$ $\bfv_2$ is $2 \times 1$ and $\bfv_3$ is $k \times 1.$ 
Then  
\begin{align*}
  U_{11} \bfv_1 + U_{12} \bfv_2 &= [1]_{8 \times 1} \\
  U_{21} \bfv_1 + U_{22} \bfv_2  + [1]_{2 \times k} \bfv_3 &= [1]_{2  \times 1}. 
\end{align*}
As $[1]_{2 \times k} \bfv_3  = ( \sum_{i=11}^{10 + k} v_i)
[1]_{2 \times 1},$  we  can rewrite this system as
\begin{align}
  U_{11} \bfv_1 + U_{12} \bfv_2 &= [1]_{8 \times 1}  \label{U1b}\\
  U_{21} \bfv_1 + U_{22} \bfv_2  &= \left(1 - \sum_{i=11}^{10 + k} v_i
  \right) [1]_{2
    \times 1}.   \label{U2b}
\end{align}
 Solving Equation \eqref{U1b} for $[ \begin{smallmatrix}  \bfv_1
  \\ \bfv_2 \end{smallmatrix}],$ we get 
\[  \begin{bmatrix}  \bfv_1
  \\ \bfv_2  \end{bmatrix} =  \bfv_0 + t_1 \bfw_1 + t_2 \bfw_2 + t_3\bfw_3,  \]
where
\begin{align*}
\bfv_0 & = \smallfrac{1}{11}(3,-1,3,0,4,4, 0,0,0,0)^T \\
\bfw_1 &= (-1,1,0,1,-1,-1,0,1,0,0)^T \\
\bfw_2 &= (-5,-2,6,0,-3,-3,0,0,11,0)^T \\
\bfw_3 &=  (-5,9,-5,0,-3,-3,0,0,0,11)^T.
\end{align*}
But  none of these solutions have all positive entries, a
contradiction to Theorem \ref{Uv}. 
Therefore, for all $s \in \boldR,$ the Lie algebra $\frakn_s^{9 + 2k}$  does not admit a nilsoliton inner
product. 
\end{proof}

\section{Proof of main theorem}\label{proof of main theorem}

Now we prove Theorem \ref{main thm}.

\begin{proof} Let $\widetilde \calN_n$ denote the moduli space of
  $\boldN$-graded, indecomposable nilpotent Lie algebras as described
  in Section \ref{moduli spaces}.  Let $\text{Nonsol}(n) \subseteq \widetilde \calN_n$ be the set of all
$\overline{\mu}$ in $\widetilde \calN_n$ such that $\frakn_\mu$ does
not admit a soliton inner product as described
  in Section \ref{moduli spaces}. 

We know from the results of Lauret, Will and Culma described in the
introduction  (\cite{lauret02}, \cite{will03}, \cite{culma1}, \cite{culma2})
that the set of nonsoliton Lie algebras in $\widetilde
\calN_n$ is discrete when $n \le 7.$

By Theorem \ref{dim 8 theorem}, none of the $8$-dimensional Lie algebras
defined in Definition \ref{dim 8} are soliton.
By Theorem \ref{dim 9 theorem}  none of the $9$-dimensional Lie algebras
defined in Definition \ref{dim 9} are soliton.
Theorem \ref{nonsoliton - n} implies that  none of the Lie algebras in dimensions $n \ge 10$
defined in Definition \ref{dim n+1} are soliton.

By Theorem \ref{nonisomorphic n},   no two  of the 
Lie algebras $\frakn_s^n$ and $\frakn_t^n$ 
defined in dimension $n \ge 8$  are isomorphic. 

Therefore, in each dimension $n \ge 8,$  the mapping $\gamma: \boldR \to \widetilde
\calN_n$  mapping $s \in \boldR$ to the equivalence class of the
nilpotent Lie algebra $\frakn_s^n$ 
is one-to-one, with image in $\text{Nonsol}(n).$

Hence, the
set  $\text{Nonsol}(n)$ consisting of nonsoliton Lie algebras (modulo equivalence
under isomorphism) in $\widetilde \calN_n$ is not discrete
when $n \ge 10.$
\end{proof}

\noindent{\it Acknowledgments.}  We are grateful to Raz Stowe for
illuminating discussions, and to Mike Jablonski for helpful comments
and useful suggestions. 

\bibliographystyle{alpha}
\bibliography{bibfile}
\end{document}